\documentclass{amsart}
\usepackage{latexsym,amssymb}

\newtheorem{theorem}{Theorem}[section]
\newtheorem{lemma}[theorem]{Lemma}
\newtheorem{proposition}[theorem]{Proposition}
\newtheorem{corollary}[theorem]{Corollary}
\newtheorem{example}[theorem]{Example}
\newtheorem{definition}[theorem]{Definition}

\newtheorem{method}{Method}

\newcommand{\Z}{\mathbb{ Z}}

\renewcommand{\S}{\mathfrak{S}}
\newcommand{\divides}{{\Big|}}

\newcommand{\Stab}{\mbox{Stab}}

\newcommand{\mo}{\mbox{\scriptsize $(-1)$}}

\author{Ian Hawthorn \and Yue Guo}
\title{Arbitrary Functions in Group Theory}
\begin{document}
\begin{abstract} Two measures of how near an arbitrary function between 
groups is to being a homomorphism are considered. These have properties 
similar to conjugates and commutators. The authors show that there is a 
rich theory based on these structures, and that this theory can be used to 
unify disparate approaches such as group cohomology and the transfer and to
prove theorems. The proof of the Schur-Zassenhaus theorem is recast in this 
context. We also present yet another proof of Cauchy's theorem and 
a very quick approach to Sylow's theorem.
\end{abstract}
\maketitle
\section{Introduction}

Consider an arbitrary function \( f : G \longrightarrow H \) between 
finite groups. Unless the function is a homomorphism it will fail to 
preserve group structure. However intuitively some non-homomorphisms 
`almost' preserve group structure while others completely scramble it. 
We consider measures of how nearly the group structure is preserved by 
an arbitrary function.

To see why this might be useful consider the function 
\( \mo : g \mapsto g^{-1} \) defined on a group \( G \). 
This function is a homomorphism if and only if \( G \) is abelian. 
Hence a measure of how nearly \( \mo \) preserves group structure will be
a measure of how near \( G \) is to being abelian. There are two basic 
structures which explore the extent of non-commutativity in a group. 
These are the commutator and the conjugate. Both are fundamentally
useful concepts in the study of finite groups.  We will find analogues 
of both in the theory of arbitrary functions and explore their properties.

\medskip

Our main objective in this paper is to show that arbitrary function techniques
may be useful in group theory. Many theorems in group theory can be expressed
as the statement that a homomorphism with certain properties exists. These
theorems can be approached by looking at a set of arbitrary 
functions with the desired properties, and then applying techniques from 
arbitrary function theory to hopefully locate or construct a homomorphism 
in the set.  We explore several different techniques for locating or 
constructing homomorphisms in a set of arbitrary functions with the aim of 
producing proofs of important results in group theory using this approach.

\section{The function action and Cauchy's theorem}

\begin{definition}
Consider an arbitrary function \( f : G \longrightarrow H \) between 
finite groups and let \( a \in G \). Define a new function
\[ f^a(x) = f(a)^{-1}f(ax) \]
\end{definition}

Clearly if \( f \) is a group homomorphism then \( f^a = f \). Conversely if 
\( f^a = f \) for all \( a \in G \) then \( f \) is a homomorphism.
We will call \( f^a \) the {\bf conjugate} of \( f \) by \( a \) and the 
process of producing such a conjugate will be called {\bf function conjugation}.
The reuse of such a common term can be justified on two grounds. 

Firstly function conjugation is closely related to ordinary conjugation.

\begin{example}
Consider the function \( \mo : G \longrightarrow G \) defined by 
\( g \mapsto g^{-1} \). Then \( \mo^a(x) = ax^{-1}a^{-1} \).
\end{example}

Secondly function conjugation has very similar properties to ordinary 
conjugation.  It is easy to verify that \( (f^a)^b(x) = f^{ab}(x) \) and 
\( f^a(1) = 1 \).  An arbitrary function \( f:G\longrightarrow H \) is 
said to be {\bf identity preserving} if \( f(1) = 1 \). The map 
\( f \mapsto f^a \) projects the set of all functions from \( G \) to 
\( H \) onto the subset of identity preserving functions. The projection 
is onto since  if \( f \) is identity preserving then \( f^1(x) = f(x) \). 
These statements together prove that 

\begin{proposition}
The maps \( f \mapsto f^{(a^{-1})} \) define a group action of \( G \) on 
the set of identity preserving functions from \( G \) to \( H \). 
Homomorphisms are precisely the functions which are invariant under this action.
\end{proposition} 

This suggests the following method by which one might demonstrate the 
existence of a homomorphism.

\begin{method} Consider a set of arbitrary functions and examine the orbits
under the function action. Use orbit counting and divisibility arguments to 
demonstrate the existence of an orbit of size 1.
\end{method}

Here is a proof of Cauchy's theorem demonstrating this approach.

\begin{theorem}[Cauchy]
If \( p \divides |G| \) then \( G \) has an element of order \( p \).
\end{theorem}
\begin{proof}
Consider the set \( \S \) of identity preserving functions from the
cyclic group \( \Z_p \) to \( G \). Then \( |\S| = |G|^{(p-1)} \) is
divisible by \( p \). The function action of \( \Z_p \) partitions
\( \S \) into orbits of size \( 1 \) or \( p \). The identity function 
is in \( \S \) and is a homomorphism, hence belongs to an orbit of 
size \( 1 \). Hence there must be at least \( p-1 \) other orbits of size
\( 1 \), and these are non-trivial homomorphisms from \( \Z_p \) to \( G \).
Any non-trivial element in the image of one of these has order \( p \)
proving the result.
\end{proof}

Of course Cauchy's theorem is elementary so finding yet another proof of
it is not a great achievement. Nevertheless it is encouraging to 
find such an immediate validation of the basic approach.  Furthermore the
proof is a rather nice one, well motivated by the search for a non-trivial 
homomorphism from \( \Z_p \) which will make it easier to explain to students 
than proofs which commence in less obvious ways.

\medskip

It is worthwhile considering whether our proof of Cauchy's theorem could be
extended to give a proof of Sylow's theorem. The answer is a qualified yes. 
We could construct such a proof by reasoning as follows.

Assume that \( H \le G \) is a \( p \)-subgroup with \( p \divides [G:H] \).
We can consider functions from \( \Z_p \) to \( G \) and define them to be 
equivalent if they generate the same map into cosets of \( H \). Orbit 
counting then gives a stabilised equivalence class and the image of this gives
an extension of \( H \). To make this argument work we must first show that 
the function action is well defined on equivalence classes which 
reduces to showing that \( p\divides [N_G(H):H] \). This is true
and not hard to prove, however once you have done this you can extend 
\( H \) by simply applying Cauchy's theorem to the quotient \( N_G(H)/H \) 
completely ignoring the fancy action on classes of arbitrary functions. 

While searching for a proof using arbitrary functions we seem to have tripped
over the following rather simple proof of Sylow's theorem, albeit one that 
does not use the methods which are the main focus of this paper.

\begin{theorem}[Sylow: Existence and Extension]
If \( H \) is a \( p \)-subgroup of a finite group \( G \) then 

\[ p \divides [G:H] \Rightarrow  p \divides [N_G(H):H] \Rightarrow
\mbox{ \( H \) is not a maximal \( p \)-subgroup.} \]
\end{theorem}
\begin{proof} 
The action by left multiplication of \( H \) on the cosets \( \{ gH \} \) 
stabilises a non-zero multiple of \( p \) of them by orbit counting. 
A coset is stabilised by this action if and only if it lies in \( N_G(H) \) 
hence \( p \divides [N_G(H):H] \). Applying Cauchy's theorem to 
\( N_G(H)/H \) gives the larger \( p \)-subgroup we are seeking.
\end{proof} 

Note that in the case that \( G \) is a \( p \)-group we can conclude
from this that normalisers increase in a \( p \)-group. 

\section{The Average Function and the Transfer}
\label{average_function_section}
\medskip

Instead of trying to show that a homomorphism exists by orbit counting we
can instead try to build one. The most direct approach is to start with an 
arbitrary function and attempt to average out the function action.

\begin{method}  Consider an orbit of functions \( \{ f^{a_i} \} \) under the 
function action. Construct a new function by averaging this orbit. 
\end{method}

In order to make this method work we need to consider what the function action
does to a product of functions.

\begin{proposition}
Let \( f \) and \( g \) be functions from \( G \) to \( H \). Define
a new function \( f*g : G \longrightarrow H \) by \( (f*g)(x) = f(x)g(x) \).
Then 
\[ (f*g)^a(x) = \bigl(f^a(x)\bigr)^{g(a)}g^a(x) \]
In particular if \( H \) is abelian then \( (f*g)^a = f^a*g^a \).
\end{proposition}

This leads directly and obviously to the following theorem

\begin{theorem}
\label{averagetheorem}
Let \( f : G \longrightarrow A \) be an arbitrary function into 
abelian \( A \) and let \( \{ f^{g_i} : i = 1 \ldots n \} \) be its 
orbit under the function action. Then the {\bf average function}
\( \overline{f} = 
	 f^{g_1}* f^{g_2} * \cdots * f^{g_n} : G \longrightarrow A \)
is a homomorphism from \( G \) to \( A \). 
\end{theorem}

If \( f \) is a homomorphism then clearly \( \overline{f} = f \) as we 
would expect. We next observe that the transfer map is defined in precisely 
this way. 

Consider a group \( G \), a subgroup \( H \le G \) with \( [G:H] = n\); 
a homomorphism \( \pi : H \longrightarrow A \) into an abelian group;
and a collection of coset representatives \( \{ t_i : 1 = 1 \ldots n \} \) 
for the cosets \( \{ t_iH \} \) of \( H \) in \( G \). 
Assume that \( 1 \in \{ t_i \} \). Look at the function 
\( f : G \longrightarrow A \) defined by \( f(ht_i) = \pi(h) \). 

If \( h_0 \in H \) then \( f^{h_0}(ht_i) = f(h_0)^{-1}f(h_0ht_i) 
= \pi(h_0)^{-1}\pi(h_0h) = \pi(h) = f(ht_i) \) so \( f \) is stabilised 
by \( H \) under the function action. 
Hence \( \{ f^{t_i} : i = 1 \ldots n \} \) covers the orbit of \( H \) 
with multiplicity \( m = [\Stab_G(f):H] \).

Now \( f^{t_i}(x) = f(t_i)^{-1}f(t_ix) = t_ixt_{(i)x}^{-1} \) where
\( t_ix \in Ht_{(i)x} \). We have deliberately chosen our notation to 
match that on page 285 of Robinson's book \cite{Robinson} allowing us to
directly identify the product of these terms as the transfer homomorphism

\[ \theta^*(g) = \prod_{i=1}^n t_ixt_{(i)x}^{-1} = \prod_{i=1}^n f^{t_i}(x)
= \bigl(\overline{f}(x)\bigr)^m \]

So the transfer map is a power of the average function. Hence among other
things we can immediately conclude that it is a homomorphism. This is a very 
notationally clean approach to the transfer. We can also apply the method to 
other functions into abelian groups to obtain further homomorphisms.

\section{Distributors}

The function action measures the failure of a function to be a homomorphism
in the same way that the conjugation action measures a failure to commute.
We can also measure a failure to commute using commutators. In this section
we introduce a construction analogous to commutators which measures the 
extent to which an arbitrary function preserves group structure.

\begin{definition}
Consider an arbitrary function \( f : G \longrightarrow H \) between 
finite groups. Define the {\bf\( f \)-distributor} \( [x,y;f] \)
of \( x \) and \( y \) to be
\[ [x,y;f] =   f(y)^{-1}f(x)^{-1}f(xy) = f(y)^{-1}f^x(y) \]
it follows that \( f(xy) = f(x)f(y)[x,y;f] \).
\end{definition}

The set of distributors as \( x \) and \( y \) range over \( G \)
measures the extent to which the function \( f \) fails to be a homomorphism.
The name {\bf distributor} was chosen because it measures the extent to 
which \( f \) distributes over group multiplication, and also because, in 
these days of electronic ignition, the word is in need of recycling.  

\begin{example}
Consider the function \( \mo : G \longrightarrow G \) defined by 
\( g \mapsto g^{-1} \). Then 
\[ [x,y;\mo] = yxy^{-1}x^{-1} = [y^{-1},x^{-1}] \]
\end{example}

Hence commutators are distributors for the function \( \mo \). 
Distributors are thus generalised commutators. 

\begin{proposition} 
\label{basicproduct}
Let \( f:G \longrightarrow H \) be an arbitrary function between finite
groups. Let \( x,y,z \in G \). Then 
\[ [y,z;f][x,yz;f] = [x,y;f]^{f(z)}[xy,z;f] \]
\end{proposition}
\begin{proof}
Expand \( f(xyz) \) in two different ways to obtain
\[ f(xyz) = f(x)f(yz)[x,yz;f] = f(x)f(y)f(z)[y,z;f][x,yz;f] \]
and also 
\[ f(xyz) = f(xy)f(z)[xy,z;f] = f(x)f(y)[x,y;f]f(z)[xy,z;f] \]
and the result follows.
\end{proof}

This identity is similar to the cocycle identity. Note that when
the function \( f \) is a coset traversal, the distributor is precisely 
the associated factor set. So distributors also can be viewed as a 
generalisation of cocycles and factor sets.

Proposition~\ref{basicproduct} can be used to shift a product in the first
component to a product in the second component. It is also possible to shift 
such a product into the function.

\begin{proposition}
\label{actionproduct} 
Let \( f:G \longrightarrow H \) be an arbitrary function between finite
groups. Let \( x,y,z \in G \). Then 
\[ [xy,z;f] = [x,z;f][y,z;f^x] \] 
\end{proposition}
\begin{proof}
Just expand out each side.
\[
 [xy,z;f] = f(z)^{-1}f(xy)^{-1}f(xyz)
\]
and also 
\[ \begin{array}{l}
[x,z;f][y,z;f^x] = f(z)^{-1}f(x)^{-1}f(xz)f^x(z)^{-1}f^x(y)^{-1}f^x(yz) \\[3mm]
\mbox{\hskip2em} 
                = f(z)^{-1}f(x)^{-1}f(xz)
			\left(f(x)^{-1}f(xz)\right)^{-1}
			\left(f(x)^{-1}f(xy)\right)^{-1}
			\left(f(x)^{-1}f(xyz)\right) \\[1mm]
\mbox{\hskip2em} 
                 = f(z)^{-1} f(xy)^{-1} f(xyz)
\end{array} \]
the result follows
\end{proof}

If \( a \in G \) then taking a distributor with \( a \) defines an operator 
\( D_a \) on the functions from \( G \) to \( H \) given by

\[ D_af (x) = [x,a;f] \]

Proposition~\ref{actionproduct} now tells us that \( (D_af)^g = D_af^g \).
So this proposition essentially states that these distributor operators 
commute with the conjugation action on functions.

\medskip

The set of commutators generates the derived subgroup, the smallest normal
subgroup whose quotient is abelian. We now make a similar observation about
distributors.

Let \( f : G \longrightarrow H \). Let \( f(G) \) denote 
\( \langle f(g)\mid g \in G \rangle \)  and let 
\( [G,G;f] = \langle [x,y;f]\mid x,y \in G \rangle \). 

\begin{proposition}
With the notation above \( [G,G;f] \unlhd f(G) \) and if 
\( \pi \) denotes the projection map onto the quotient then
\( \pi f \) is a homomorphism.

Furthermore if \( K \unlhd f(G) \) is such that \( \pi f \) is a homomorphism
where \( \pi \) is the projection map onto \( f(G)/K \), then
\( [G,G;f] \le K \).
\end{proposition}

\begin{proof}
Obviously \( [G,G;f] \le f(G) \). To prove normality we use  
proposition~\ref{basicproduct} which gives

\[ [x,y;f]^{f(z)} = [y,z;f][x,yz;f][xy,z;f]^{-1} \]

The rest follows by observing that 
\( \pi f(xy) = \pi f(x) \pi f (y) \pi ([x,y;f]) \) and hence \( \pi f \) is 
a homomorphism if and only if \( [x,y;f] \) always lies in the kernel of 
\( \pi \).
\end{proof}

This gives us another quite obvious method of constructing homomorphisms

\begin{method}
Given any function \( f : G \longrightarrow H \) we may construct a 
homomorphism into the group \( f(G)/[G,G;f] \).
\end{method}

While this method allows to us easily construct homomorphisms out of 
\( G \) unfortunately we cannot easily specify the group into which 
the homomorphism will map. In particular it is often difficult to ensure 
that a homomorphism constructed in this fashion is not trivial.
This is a rather severe limitation on this method as a potential tool in 
proving group theoretic results.  

\section{The Distributed Average and Schur-Zassenhaus}

In section~\ref{average_function_section} we built a homomorphism from
an arbitrary function \( f \) by averaging out the effects of the function
action. In order to be able to take an average we required that \( f \) map 
into  an abelian group, which is a rather strong restriction on our ability 
to apply the method.  

However if \( f \) is already close to being a homomorphism, we shouldn't 
need to adjust the part of \( f \) that is already behaving itself. 
It is only the part that is misbehaving that needs to be averaged.  
The distributor \( [a,x;f] = f(x)^{-1}f^a(x) \) describes the difference 
between the function \( f \) and its conjugate \( f^a \) and hence describes 
only this misbehaving part of the function.

Instead of averaging the entire function perhaps we should try averaging the
distributors. By combining the average distributor with \( f \) we may hope to
obtain a homomorphism. And since we are only averaging distributors we will 
only need \( [G,G;f] \) to be abelian which in most cases will be a weaker 
restriction. 

\begin{method} Combine the average distributor with the function to obtain
a homomorphism.
\end{method}

Let \( f : G \longrightarrow H \) be an arbitrary function and assume that
\( [G,G;f] \) is abelian. Let \( \{ f^{a_i} \} \) be the set of all 
conjugates of \( f \) and fix elements \( \{ a_i: 1 = 1 \ldots n \} \)
giving these conjugates. Note that \( \{ a_i \} \) is a set of coset
representatives for \( \Stab_G(f) \) in \( G \).  Consider

\[ \prod_i \left(f(x)^{-1}f^{a_i}(x)\right) = \prod_i [a_i,x;f] \]

This is not really the average distributor. It is really just the product 
of all the distributors and will be the \( n \)th power of what we might 
think of as the true average where \( n = [G:\Stab_G(f)] \).

Sometimes a product is good enough. Indeed we used a product rather than a 
true average to define the average function and obtain the transfer map. 
In this case however we need a true average since we plan to recombine 
the result with the original function \( f \). To obtain a true average we 
must take an \( n \)th root which is only possible if the number of conjugates 
\( n = [G:\Stab_G(f)] \) of \( f \) is relatively prime to \( |[G,G;f]| \). 
Under this extra assumption we may find a number \( m \) with 
\( mn = 1 \pmod{|[G,G;f]|} \) and the true average distributor will then be 

\[ d(x) = \left( \prod_i [a_i,x;f] \right)^m \]

We recombine this with \( f \) to form the new function

\[ \overline{\overline{f}}(x) = f(x)d(x) \]

which we will call the {\bf distributed average} of \( f \). Note that if 
\( f \) is a homomorphism then clearly \( \overline{\overline{f}}(x) = f(x) \) 
as we would expect.

\medskip
Before continuing we address a small technical matter. The definition of
the distributed average above depends on \( \Stab_G(f) \le G \) and on 
\( [G,G;f] \le H \). However in practice we may not know enough about the 
function \( f \) to explicitly determine these subgroups. Our next Lemma 
shows that the distributed average can be computed without determining 
these subgroups explicitly.

\begin{lemma}
\label{TechnicalLemma} Let \( f : G \longrightarrow H \) be a function. Let 
\( [G,G;f] \le A \) where \( A \) is abelian, and let \( K \le \Stab_G(f) \)
with \( \mbox{gcd}([G:K],|A|) = 1 \). Let \( m.[G:K] = 1 \pmod{|A|} \). 
Choose  a set \( \{ a_i \} \) of coset representatives for \( K \) in 
\( G \). Then

\[ \overline{\overline{f}}(x) = f(x) \left( \prod_i [a_i,x;f] \right)^m \]

and is insensitive to the choice of subgroups \( K \) and \( A \) 
and the choice of coset representatives.
\end{lemma}
 
\begin{proof}
The choice of \( A \) only influences the choice of \( m \) via the equation
\( m.[G:K] = 1 \pmod{|A|} \). But any such \( m \) also satisfies
\( m.[G:K] = 1 \pmod{|[G,G;f]|} \). Note that \( m \) is applied as a power
to elements of \( [G,G;f] \). So any number \( m \) satisfying this condition 
will give the same result.

The coset representatives only enter into the definition in the terms
\( [a_i,x;f] = f(x)^{-1}f^{a_i}(x) \). But \( f^{a_i} = f^{a'_i} \) 
whenever \( a_i \) and \( a'_i \) belong to the same coset of 
\( \Stab_G(f) \) (which will definitely be true if they belong to 
the same coset of \( K \)). Hence the definition does not depend on
the choice of \( \{ a_i \} \). 

Now each coset of \( \Stab_G(f) \) consists of \( [\Stab_G(f):K] \)
cosets of \( K \). So the effect of using \( K \) instead of \( \Stab_G(f) \)
in the definition is to simply apply a power of \( [\Stab_G(f):K] \) to the 
product of distributors. But since \( [\Stab_G(f):K] \) is relatively prime 
to \( |A| \) this extra power will be taken account of in our choice of 
the power \( m \) and the result will remain the same. 

Hence \( \overline{\overline{f}} \) is well defined and insensitive to 
the choice of the subgroups \( A \) and \( K \).
\end{proof}

\begin{theorem}\label{DAhomomorphism}
 With the notation and under the conditions discussed above, 
the distributed average \( \overline{\overline{f}} \) is a homomorphism from \( G \) 
to \( H \).
\end{theorem}
\begin{proof} 
We directly calculate as follows
\[ \begin{array}{rl}
\overline{\overline{f}}(xy) &=
f(xy) \left( \prod_i [a_i,xy;f] \right)^m \\
&= f(x)f(y)[x,y;f] 
\left( \prod_i [x,y;f]^{-1}f(y)^{-1}[a_i,x;f]f(y)[a_ix,y;f] \right)^m \\
&= f(x)f(y)[x,y;f][x,y;f]^{-nm}
f(y)^{-1} \left( \prod_i [a_i,x;f]\right)^m f(y)
\left( \prod_i [a_ix,y;f] \right)^m \\
&= f(x)d(x)f(y)d(y) \\
&= \overline{\overline{f}}(x)\overline{\overline{f}}(y)
\end{array}
\]
\end{proof}

\begin{corollary}
\label{SZabelian}
Assume that \( f : G \longrightarrow H/A \) is a homomorphism where 
\( A \) is abelian and \( |A| \) is prime to \( |G| \). Then we may 
lift \( f \) to obtain a homomorphism into \( H \). 
\end{corollary}
\begin{proof}
Any choice of coset representatives defines a function 
\( \hat{f}: G \longrightarrow H \) with \( f(g) = \hat{f}A \).
The function \( \hat{f} \) satisfies the conditions of 
theorem~\ref{DAhomomorphism} and the homomorphism 
\( \overline{\overline{\hat{f}}} \) is the desired lifting of \( f \).
\end{proof}

\begin{corollary}
\label{SZsoluble}
Assume that \( f : G \longrightarrow H/N \) is a homomorphism where 
\( N \) is soluble and \( |N| \) is prime to \( |G| \). Then we may 
lift \( f \) to obtain a homomorphism into \( H \). 
\end{corollary}
\begin{proof}
Since \( N \) is soluble we may decompose the projection 
\( H \longrightarrow H/N \) into a sequence of projections

\[ H = H_0 \longrightarrow H_1 \longrightarrow H_2 
\longrightarrow \cdots \longrightarrow H_k = H/N \]
where the kernel of each projection is abelian. Repeatedly applying 
the previous corollary we may lift the function \( f \) to obtain
eventually a lifting of the homomorphism into \( H \).
\end{proof}

We could also have used Schur-Zassenhaus to prove these corollaries. Indeed 
corollary~\ref{SZabelian} is pretty much equivalent to the abelian case of
that theorem which is the most difficult part to prove. The abelian case 
of the Shur-Zassenhaus theorem is usually proved via a cohomological argument 
using cocycles and factor sets. Unwrapping the notation we find that the
usual proof is at its heart the same as the distributed average approach
presented above. However the distributed average approach is much better 
motivated and easier to understand.

\medskip

The Schur-Zassenhaus theorem also states that complements are unique up to 
conjugacy. This motivates us to look at the question of uniqueness.  

The distributed average function \( \overline{\overline{f}} \) is of course
unique for any given function \( f \), so this is not where the
question of uniqueness arises. In the proof of Corollary~\ref{SZabelian} 
however we took the distributed average of a function \( \hat{f} \) which 
was defined via an arbitrary choice of coset representatives. Hence the 
question we must address is what effect the choice of coset representatives 
(and thus the function \( \hat{f} \)) has on the distributed average. This 
leads us to consider the following situation

Suppose \( f : G \longrightarrow H \) with \( [G,G;f] \le A \) with 
\( A \unlhd H \) abelian; and assume \( K \le \Stab_G(f) \) with
\( \mbox{gcd}([G:K]:|A|) = 1 \). Let
\( a: G \longrightarrow A \) be any function with \( K \le \Stab_G(a) \).
Consider the function \( (f*a)(g) = f(g)a(g) \). We are interested in the 
relationship between \( \overline{\overline{f*a}} \) and 
\( \overline{\overline{f}} \). Now

\[ \begin{array}{rl}
[t,x;f*a] &= a(x)^{-1}f(x)^{-1}a(t)^{-1}f(t)^{-1}f(tx)a(tx) \\
          &= f(x)^{-1}a(t)^{-1}f(x)a(t)[t,x;f][t,x;a]
\end{array} \]

Hence 

\[ \begin{array}{rl}
\overline{\overline{f*a}} 
&= \displaystyle f(x)a(x)\Bigl(
	\prod_{t_iK}f(x)^{-1}a(t_i)^{-1}f(x).a(t_i).[t_i,x;f][t_i,x;a]
		\Bigr)^m  \\
&= \displaystyle f(x)
	a(x) 
	\Bigl(f(x)^{-1}A^{-1}f(x)\Bigr)
	A
	\Bigl(\prod_{t_iK}[t_i,x;f]\Bigr)^m
	\Bigl(\prod_{t_iK}[t_i,x;a]\Bigr)^m \\
&= \displaystyle f(x)
	\Bigl(f(x)^{-1}A^{-1}f(x)\Bigr)
	\Bigl(\prod_{t_iK}[t_i,x;f]\Bigr)^m
	A\,
	a(x)
	\Bigl(\prod_{t_iK}[t_i,x;a]\Bigr)^m \\
&= \displaystyle A^{-1}\overline{\overline{f}}(x) A\, \overline{\overline{a}}(x)
\end{array} \]

Where \( \displaystyle A = \Bigl(\prod_{t_iK}a(t_i)\Bigr)^m \). 

Finally observe that \( \overline{\overline{a}} \) is a homomorphism 
from \( G \) into \( A \). Since the function \( a \) is stabilised by 
elements of \( K \), \( \overline{\overline{a}} \) acts trivially on 
\( K \). Hence the kernel of \( \overline{\overline{a}} \) contains 
\( K \). So by the fundamental theorem of homomorphisms the order of
its image is a factor of \( [G:K] \) and by Lagrange it divides \( |A| \).
Since these are relatively prime we conclude that \( \overline{\overline{a}}
\) is the trivial homomorphism.

We have proved the following theorem.

\begin{theorem}
Under the conditions given above
\[  \overline{\overline{f*a}}(x)  = 
\left(\overline{\overline{f}}(x)\right)^A 
\mbox{ \  where \ }
A = \Bigl(\prod_{t_iK}a(t_i)\Bigr)^m \]
\end{theorem}

\begin{corollary}
The lifted functions in Corollaries~\ref{SZabelian}~and~\ref{SZsoluble} 
are unique up to conjugacy in \( N \).
\end{corollary}

\begin{proof} Let \( f : G \longrightarrow H/A \) be a homomorphism and
let \( f_i : G \longrightarrow H \) for \( i = 1,2\) be homomorphisms which 
lift \( f \). Then \( f_1=f_2*a \) for some function 
\( a : G \longrightarrow A \). As they are homomorphisms we have 
\( f_i = \overline{\overline{f_i}} \) and the previous theorem then 
tells us that \( f_1(x) =\left(f_2(x)\right)^A \) where 
\( A = \left(\prod_{t_iK}a(t_i)\right)^m \) proving the result for 
Corollary~\ref{SZabelian}.

The result for Corollary~\ref{SZsoluble} follows by induction on the order
of the soluble normal subgroup \( N \).
\end{proof}

\section{Conclusion}

Arbitrary function theory has potential as a unifying concept in group theory.
Proofs approached in this manner are in some cases more direct, more obviously 
motivated, and simpler to understand than proofs using concepts such as 
cohomology which have been imported into group theory from elsewhere in 
mathematics. The question of whether this approach could be used in proofs
currently requiring representation theory is worthy of investigation.

\end{document}